\documentclass[a4paper,12pt,openbib,reqno]{amsart}
\usepackage{amsmath,amsfonts,amssymb}
\usepackage{verbatim}
\usepackage{enumerate}
\usepackage{url}
\usepackage[colorlinks]{hyperref}
\usepackage[latin1]{inputenc}
\usepackage[english]{babel}
\usepackage{mathrsfs}
\usepackage{tikz}
\usepackage[margin=2.5cm]{geometry}
\def\N{\mathbb N}

\def\Z{\mathbb Z}

\def \F{\mathbb F}

\def\Fx{\mathbb{F}_q[x]}
\def\Fq{\mathbb{F}_q}
\def\Fp{\mathbb{F}_p}

\def\Fpn{\mathbb{F}_{p^n}}

\def\ord{\mathop{\rm ord}\nolimits}

\theoremstyle{plain}
\newtheorem{theorem}{Theorem}[section]
\newtheorem{lemma}[theorem]{Lemma}
\newtheorem{definition}[theorem]{Definition}
\newtheorem{corollary}[theorem]{Corollary}
\newtheorem{proposition}[theorem]{Proposition}
\newtheorem{remark}[theorem]{Remark}
\newtheorem{example}[theorem]{Example}

\def\qed{\hfill\hbox{$\square$}}
\DeclareMathOperator{\ind}{ind}

\theoremstyle{definition}

\author[J. Alves Oliveira]{Jos\'e Alves Oliveira}
\author[F. E. Brochero Mart\'{\i}nez]{F. E. Brochero Mart\'{\i}nez}
\address{
	Departamento de Matem\'{a}tica\\
	Universidade Federal de Minas Gerais\\
	UFMG\\
	Belo Horizonte, MG\\
	31270-901\\
	Brazil\\
}
\email{jose-alvesoliveira@hotmail.com}
\email{fbrocher@mat.ufmg.br }

\title{Permutation Polynomials with Carlitz Rank 2}

\date{\today
}

\subjclass[2000]{ }


\begin{document}
	
	\begin{abstract} Let $\Fq$ denote the finite field with $q$ elements. The Carlitz rank of a permutation polynomial is an important measure of complexity of a polynomial. In this paper we find a sharp lower bound  for the weight of any permutation polynomial with Carlitz rank 2,   improving  the bound found by G\'omez-P\'erez, Ostafe and Topuzo\u{g}lu in that case.
	\end{abstract}
	
	\maketitle
	
	\section{Introduction}
	Let $\Fq$ denote the finite field with $q$ elements. A polynomial $f(x)\in\Fx$ is called a \textit{permutation polynomial} over $\Fq$ if the map $a\mapsto f(a)$ permutes the elements of $\Fq$. Important early contributions to the general theory can be found in Hermite \cite{hermite1863fonctions} and Dickson \cite{dickson1896analytic}. Recently, the study of permutation polynomials has intensified by their applications in cryptography and coding theory \cite{dillon2004new,colbourn2004permutation,sun2003permutation,chu2002circular}, resulting in the emergence of many new classes of permutation polynomials. Polynomials with few non-null coefficients are of high interest for cryptography and they have been extensively explored (see \cite{masuda2009permutation,hou2015determination,hou2015determination2, li2017new}).
	
	Let $\alpha$ be a primitive element of $\Fq^*$ and let $s_n=f(\alpha^n)$ be a sequence of period $q-1$. Blahut's Theorem (see Theorem 10.4.29 in \cite{Panario}) states that the linear complexity of the sequence $\{s_n\}_{n\in\N}$ is the weight of $f$. Our aim in this paper is to prove that the weight of a polynomial is close to $q$ if its Carlitz rank is equal to $2$, providing then a lower bound for the linear complexity of the sequence $\{s_n\}_{n\in\N}$. Although the small Carlitz rank does not suggests this sequence for cryptography, the estimate that we present for the weight of $f$ suggest that the sequence $s_n$ is an attractive candidate for Monte-Carlo methods, see Theorem 10.4.87 (with $n=2(q-1)$) in \cite{Panario}.  
	
	Two polynomials $f(x),g(x)\in\Fx$ represent the same permutation of $\Fq$ if $f(a)=g(a)$ for all $a\in\Fq$, i.e. $f(x)\equiv g(x)\pmod{x^q-x}$.  Let $\mathcal{S}_q$ be the set of permutation polynomials of $\Fq$. It is well-known (see Carlitz \cite{carlitz1953permutations})  that $\mathcal{S}_q$ is generated by linear polynomials $ax+b$, with $a,b\in\Fq,a\neq0$, and $x^{q-2}$. Therefore any permutation polynomial $f$ of $\Fq$ can be represented by a polynomial of the form
	\begin{equation}\label{equa1} 
	\mathcal{P}_n(x) = (\dots((a_0x+a_1)^{q-2}+a_2)^{q-2}\ldots+a_n)^{q-2}+a_{n+1}
	\end{equation}
	with $a_1, a_{n+1}\in\Fq$ and $a_0, a_2,\dots, a_n\in\Fq^*$, i.e. $f(x)\equiv\mathcal{P}_n(x)\pmod{x^q-x}$. Indeed, the polynomial $f$ can be represented in more than one way  by  polynomials of the form \eqref{equa1} and then the following invariant of permutation polynomials was introduced in \cite{aksoy2009carlitz}.

	\begin{definition} Let $f$ be a permutation polynomial of\, $\Fq$. The smallest integer $n$ for which there exists a polynomial $\mathcal{P}_n(x)$ of the form \eqref{equa1} such that $f(x)\equiv\mathcal{P}_n(x)\pmod{x^q-x}$ is called Carlitz rank of $f$. We denote by $Crk(f)$ the Carlitz rank  of the polynomial $f$.
	\end{definition}

	Recently, many authors have been working on problems concerning Carlitz rank, e.g see~\cite{icsik2018carlitz,icsik2017complete,gomez2014carlitz}. For a survey of Carlitz rank, see \cite{topuzouglu2014carlitz}.	For a permutation polynomial  $f$ with $Crk(f)=n$, let $\mathcal{P}_n(x)$ be a polynomial  representation of  $f$ of the form  (\ref{equa1}). Since
	$$x^{q-2}=\begin{cases}
	x^{-1},&\text{ if }x\neq0;\\
	0,&\text{ if }x=0,
	\end{cases}$$
we can rewrite $\mathcal{P}_n(x)$  formally as
	\begin{equation}\label{equa2} 
	a_{n+1}+\frac{1}{\ \ a_n+\ \dots\ \cfrac{1}{\ \ a_2+\cfrac{1}{a_0x+a_1}\ \ }\ \ },
	\end{equation}\\[0.3em]
	and its  $n$th convergent as 
	\begin{equation}\label{equa3}  
	\mathcal{R}_n(x) = \frac{\alpha_{n+1}x+\beta_{n+1}}{\alpha_{n}x+\beta_{n}},
	\end{equation}
	where the sequences $\{\alpha_n\}$ and $\{\beta_n\}$ are inductively defined as   $\alpha_k=\alpha_{k-1}a_k+\alpha_{k-2}$ and $\beta_k=\beta_{k-1}a_k+\beta_{k-2}$ for $k\geq 2$ with $\alpha_0=0$, $\alpha_1=a_0$, $\beta_0=1$ and $\beta_1=a_1$.
	
	Let $\mathcal{O}_n$ denote the set of poles
	\begin{equation}\label{equa4}
	 \Bigg\{\frac{-\beta_i}{\alpha_i} : i=1 \ldots n\Bigg\}\subset\mathbb{P}^1(\Fq)=\Fq\cup\{\infty\}.
	\end{equation}
	We observe that  $\mathcal{P}_n(x)=\mathcal{R}_n(x)$ for all $x\in\Fq\backslash\mathcal{O}_n$ and therefore
	\begin{equation}\label{equa5}
	f(x) = \frac{\alpha_{n+1}x+\beta_{n+1}}{\alpha_{n}x+\beta_{n}}\text{ for all }x\in\Fq\backslash \mathcal{O}_n.
	\end{equation}
	
	In \cite{aksoy2009carlitz}  the authors show the following relation between  Carlitz rank  and the degree of a permutation polynomial. 
	
	\begin{proposition}\label{item29}Let $f(x)$ be a permutation polynomial of degree $d$ and Carlitz rank $n$. Then
		$$n\geq q-1-d.$$
	\end{proposition}
	
	 For a polynomial $f\in\Fq[x]$, let $\omega(f)$ be the weight of $f$, i.e.  the number of non-null coefficients of $f$. The following important result,  shown  by G\'omez-P\'erez, Ostafe, and  Topuzo\u{g}lu,  relates  Carlitz rank with the weight of a permutation polynomial.
	
	\begin{theorem}\cite[Theorem $4$]{gomez2014carlitz}\label{item90}
		Let $f$ be a permutation polynomial of $\Fq$ with $\deg(f)\geq2$. Suppose that $f$ has a representation of the form 
		$$f(x)=\sum\limits_{i=1}^{\omega(f)}a_ix^{e_i},$$
where $f(x)\neq c_1+c_2x^{q-2}$  with $c_1,c_2\in\Fq, c_2\neq0$. Then
		$$ Crk(f) > \frac{q}{\omega(f)+2}-1.$$
	\end{theorem}
	
	For $Crk(f)=n$, this theorem entails that the weight of $f$ has a lower bound given by
	\begin{equation}\label{equa122}
	\omega(f)>\frac{q}{n+1}-2.
	\end{equation}

 In this paper, we improve this lower bound for the weight of a permutation polynomial $f$ in the case where the Carlitz rank of $f$ is $2$ (Theorem \ref{item34}). In particular, we prove that $\omega(f)\ge q-\tfrac{q}{p}+O(p^{1/2})$ provided $Crk(f)=2$.
	
	\section{ Preliminaries}
	
	In this section we provide some lemmas that will be used to prove our main results. Throughout this article, $\F_q$ denotes the finite field with $q$ elements, where $q$ is a power of a prime $p$. For any $a\in \F_q^*$, let $\ord_{\Fq}(a)$ be the order of $a$ in the cyclic group $\F_q^*$. 
	
\begin{lemma}\label{item30}	
	Let  $f(x)\in\Fx$ be a permutation polynomial of $\Fq$ with $Crk(f)=2$. Then there exist elements $a_0, a_1, a_2, a_3\in\Fq$, with $a_0\neq0$ and $a_2\neq 0$, such that		
	$$f(x)\equiv a_2^{-1}\sum\limits_{i=1}^{q-2}{x^i(-a_0)^i\big[(a_1-\!ia_2^{-1})(a_1+a_2^{-1})^{q-2-i}- a_1^{q-1-i}\big]}+c\pmod{x^q-x},$$
	where $c= a_3+a_2^{-1}\left[a_1(a_1+a_2^{-1})^{q-2}+1-a_1^{q-1}\right]$.	\end{lemma}
\begin{proof}
		By definition of Carlitz rank, there exist $a_0, a_1, a_2, a_3\in\Fq$ with $a_0\neq0$ and $a_2\neq 0$ such that $f(x)=((a_0x+a_1)^{q-2}+a_2)^{q-2}+a_3$. By Equation \eqref{equa5}, we have that 
$f\big(a_0^{-1}x\big)\equiv g(x) \pmod {x^q-x}$, where
		\begin{equation}\label{item66}
		g(x)=\begin{cases}
		\ \mathcal{R}_2(x)=\dfrac{x+a_1}{a_2x+a_1a_2+1}+a_3,& \text{ if }x \notin \{-a_1, -a_1-a_2^{-1}\};\\[.9em]
		\ a_2^{-1}+a_3,& \text{ if }x = -a_1;\\[.5em]
		\ a_3,& \text{ if }x = -(a_1+a_2^{-1}).
		\end{cases}
		\end{equation}
		
		Moreover, if $x \neq -(a_1+a_2^{-1})$ it follows that
		\begin{equation}\label{item21}
		\begin{aligned}
		\dfrac{x+a_1}{a_2x+a_1a_2+1}+a_3 &\equiv a_2^{-1}(x+a_1)\big(x+a_1+a_2^{-1}\big)^{q-2}+a_3\pmod{x^q-x}\\[.5em]
		&= a_2^{-1}(x+a_1)\sum\limits_{i=0}^{q-2}{{q-2 \choose i}x^i\big(a_1+a_2^{-1}\big)^{q-2-i}}+a_3=:\overline{\mathcal{R}}_2(x).
		\end{aligned}
		\end{equation}
		
		Therefore $f\big(a_0^{-1}x\big)-\overline{\mathcal{R}}_2(x)=0$ for each element  $x\in\Fq\backslash\{-a_1, -a_1-a_2^{-1}\}$. On the other hand, using the Lagrange's Interpolation Method, the polynomial $f\big(a_0^{-1}x\big)-\overline{\mathcal{R}}_2(x)$ can  be written  as
		\begin{equation}\label{item20}
		\begin{aligned} f\big(a_0^{-1}x\big)-\overline{\mathcal{R}}_2(x)&\equiv \sum\limits_{a\in\Fq}{\big[f(a_0^{-1}a)-\overline{\mathcal{R}}_2(a)\big](1-(x-a)^{q-1})}\pmod{x^q-x}\\
		&=(1-(x+a_1)^{q-1})a_2^{-1}+(1-(x+a_1+a_2^{-1})^{q-1})\cdot 0 \\
		&=(1-(x+a_1)^{q-1})a_2^{-1}.\\		
		\end{aligned}
		\end{equation}
		By Eq.~\eqref{item21} and \eqref{item20}, we have that
		$$\begin{aligned} f\big(a_0^{-1}x\big)&=\left(\tfrac{x+a_1}{a_2}\right)\sum\limits_{i=0}^{q-2}{{q-2 \choose i}x^i(a_1+a_2^{-1})^{q-2-i}}+a_3+(1-(x+a_1)^{q-1})a_2^{-1}\\
		&=\left(\tfrac{x+a_1}{a_2}\right)\sum\limits_{i=0}^{q-2}{{q-2 \choose i}x^i\eta^{q-2-i}}+a_3+a_2^{-1}-a_2^{-1}\sum\limits_{i=0}^{q-1}{{q-1 \choose i}x^i a_1^{q-1-i}},\\
		\end{aligned}$$
		where $\eta:=a_1+a_2^{-1}$. From Lucas' congruence it follows that
		$$\begin{aligned} f\big(a_0^{-1}x\big)&\equiv \left(\tfrac{x+a_1}{a_2}\right)\sum\limits_{i=0}^{q-2}{(i+1) (-x)^i\eta^{q-2-i}}+a_3+a_2^{-1}-a_2^{-1}\sum\limits_{i=0}^{q-1}{(-x)^i a_1^{q-1-i}}\hskip-0.4cm\pmod{x^q-x}\\
		&=a_2^{-1}\sum\limits_{i=1}^{q-2}(-x)^i\left[(a_1-ia_2^{-1})\eta^{q-2-i}- a_1^{q-1-i}\right]+c,	\\		
		\end{aligned}$$
		where $c:= a_3+a_2^{-1}\left[a_1(a_1+a_2^{-1})^{q-2}+1-a_1^{q-1}\right]$. Therefore
		$$f(x)\equiv a_2^{-1}\sum\limits_{i=1}^{q-2}{(-a_0x)^i\left[(a_1-ia_2^{-1})(a_1+a_2^{-1})^{q-2-i}- a_1^{q-1-i}\right]}+c\pmod{x^q-x},$$
	from where our result follows. $\hfill\qed$
	\end{proof}

\begin{remark}\label{item40}
	For a polynomial $f$ with Carlitz rank $2$, let $a_0$ and $c$ be as defined in Lemma~\ref{item30}. Since our goal in this paper is to present a lower bound for the weight of $f$, we can assume without loss of generality that $a_0=-1$ and $c=0$.
\end{remark}

\begin{remark}\label{item39}
	We recall that the needed for $a_2$ being non-null in Lemma~\ref{item30} follows from the definition of Carlitz rank. Otherwise, the polynomial $f$ actually has Carlitz rank $1$.
\end{remark}

Using the last lemma  we see that it is necessary  to determine the elements $a_1$ and $a_2$ for which the relation  
\begin{equation}\label{item17}
(a_1-ia_2^{-1})(a_1+a_2^{-1})^{q-2-i} - a_1^{q-1-i}=0
\end{equation}
 has  the largest  number of solutions $i$ with $ 1\leq i\leq q-2$. If either $a_1+a_2^{-1}= 0$ or $a_1= 0$, then it is easy to compute the exact number of solutions of the Equation~\eqref{item17}. If $a_1+a_2^{-1}\neq 0$ and $a_1\neq 0$, then we want to estimate the number os solutions of the equation
	\begin{equation}\label{equa597}
	a_1-ia_2^{-1}=\left(\frac{a_1+a_2^{-1}}{a_1}\right)^i(a_1+a_2^{-1}).
	\end{equation}
	
	In the following results we provide the necessary theory to obtain an upper bound on the number of solutions of \eqref{equa597}.

	\begin{lemma}\label{item69}
		Let $\Omega$ be a set and  let $g_1,g_2:\Z\rightarrow\Omega$ be periodic functions with period $n_1$ and $n_2$ respectively. For $u\in\Omega$, set $m_{i}(u)=|\{j\in[1,n_i]:g_i(j)=u\}|$ with $i\in\{1,2\}$. If $\gcd(n_1,n_2)=1$, then
		$$ \left|\left\{i\in[1, n_1 n_2]: g_1(i)=g_2(i)\right\}\right| = \sum_{u\in\Omega} m_1(u)m_2(u).$$
	\end{lemma}
	\begin{proof}We observe that there exist $m_1(u)$ integers $j_1\in[1,n_1]$ such that $g_1(j_1)=u$. Similarly, there exist $m_2(u)$ integers $j_2\in[1,n_2]$ such that $g_2(j_2)=u$. Therefore, there exist $m_1(u)m_2(u)$ pairs $(i_1,i_2)\in[1,n_1]\times[1,n_2]$ such $g_1(i_1)=g_2(i_2)=u$. Furthermore, if $j_1\in[1,n_1]$ and $j_2\in [1,n_2]$, then by the Chinese Remainder Theorem there exists an unique $j\in[1,n_1n_2]$ such that $j\equiv j_1\pmod{n_1}$ and $j\equiv j_2\pmod{n_2}$. Therefore, there exist exactly $m_1(u)m_2(u)$ values $j\in[1,n_1n_2]$ such that $j\equiv i_1\pmod{n_1}$, $j\equiv i_2\pmod{n_2}$ and $g_1(j)=g_2(j)=u$. Set
		$$d_u:=|\{i\in[1, n_1n_2]: g_1(i)=g_2(i)=u\}| = m_1(u)m_2(u).$$
		Our result follows by noting that $\left|\left\{i\in[1, n_1 n_2]: g_1(i)=g_2(i)\right\}\right|=\sum_{u\in\Omega} d_u$. $\hfill\qed$ 
		
	\end{proof}

	\begin{corollary}\label{item31}
		Let $\Omega$ be a set and $l,k$ be nonnegative  integers. Let $g_1,g_2:\Z\rightarrow\Omega$ be periodic functions with period $n_1$ and $n_2$, respectively. If $g_1|_{[1,n_1]}$ and $g_2|_{[1,n_2]}$ are injective functions and  $\gcd(n_1,n_2)=1$, then
		$$ |\{i\in[k+1, k+ln_1n_2]: g_1(i)=g_2(i)\}| \leq l \times\min\{n_1,n_2\}.$$
	\end{corollary}

An approach on the number of solutions of the Equation~\eqref{equa597} in the case where $q=p$ can be found in Theorem 1 (with $n=1$) of Coppersmith and Shparlinski~\cite{coppersmith2000polynomial}, where the authors estimate the number of solutions of $\ind(x)\equiv f(x)\pmod{p}$ with $f\in\Fq[x]$ and $n=\deg(f)$, where $\ind(x)$ denotes the index of $x$ with respect to a fixed primitive element of $\Fq^*$. In our case, the bound presented in \cite{coppersmith2000polynomial} implies that the number of solutions of the Equation~\eqref{equa597} is bounded by $\sqrt{2p-31/4}+1/2$. The following result yields a tighter bound for the case where $n=1$. Furthermore, we extend the result to the case where $q=p^t$ and $\ind(x)$ is inside a box whose size is bounded by $p$, improving Coppersmith and Shparlinski's result in the case $n=1$.
	
\begin{lemma}\label{item32}
Let $\gamma, c, d\in\Fq$ with $c\neq0$. If $L$ and $M$ are integers such that $3\leq M\leq p$, then
		$$|\{L\leq i\leq L+M: \gamma^{i+1}= ic+d \}|\leq \sqrt{\frac{3M}{2}-\frac{39}{16}}+\frac{5}{4}.$$
\end{lemma}

	\begin{proof}
		For $\gamma\in\{0,1\}$ the inequality is trivial. Let $\gamma\in\Fq\backslash\{0,1\}$ and define $$\mathscr{C}_{\gamma}=\{L\leq i\leq L+M:\ \gamma^{i+1}= ic+d\},\quad  t=|\mathscr{C}_{\gamma}|\quad\text{and}\quad l=ord_{\Fq}(\gamma).$$ 
		
Suppose that  $i_1, i_2$ are distinct elements of  $\mathscr{C}_{\gamma}$ with  $\gamma^{i_1+1}= i_1c+d$ and $\gamma^{i_2+1}= i_2c+d$. These two equations entails that
		$$\gamma^{i_2+1}-\gamma^{i_1+1} = i_2c-i_1c$$
		$$\gamma^{i_2+1}-\gamma^{i_1+1} = (i_2-i_1)c$$
		$$\gamma^{i_1+1}(\gamma^{i_2-i_1}-1) = (i_2-i_1)c.$$
		We observe that if $l|(i_2-i_1)$ then $(i_2-i_1)c=0$ and therefore  $i_1=i_2$,  which is a contradiction. Then   $i_1\not\equiv i_2\pmod{l}$  and 
		$$\gamma^{i_1+1} = (i_2-i_1)\frac{c}{\gamma^{i_2-i_1}-1}.$$
		Now, suppose that there exists elements $j_1, j_2\in \mathscr{C}_{\gamma}$ with $j_1\neq j_2$ such that $j_2-j_1=i_2-i_1$. Then 
		$$j_1c+d = \gamma^{j_1+1} = (j_2-j_1)\frac{c}{\gamma^{j_2-j_1}-1} = (i_2-i_1)\frac{c}{\gamma^{i_2-i_1}-1} = \gamma^{i_1+1} = i_1c+d.$$
		
		Since $c\neq 0$, it follows that $j_1=i_1$ and $j_2=i_2$. Therefore, the difference between two distinct pairs of elements in $\mathscr{C}_{\gamma}$ is never the same. In particular, if $\mathscr{C}_{\gamma}=\{i_1< \dots <i_{t}\}$, then the values
		$$(i_2-i_1),(i_3-i_2),\dots,(i_t-i_{t-1}),$$
		$$(i_3-i_1),(i_5-i_3),\dots,(i_{2\lfloor \frac{t-1}{2}\rfloor+1}-i_{2\lfloor \frac{t-1}{2}\rfloor-1}),$$
		$$(i_4-i_2),(i_6-i_4),\dots,(i_{2\lfloor \frac{t}{2}\rfloor}-i_{2\lfloor \frac{t}{2}\rfloor-2})$$
		are all distinct. The number of values in the list above is $2t-3$. Furthermore,
		$$M_1:=(i_2-i_1)+(i_3-i_2)+\dots+(i_t-i_{t-1})\leq M-1,$$
		$$M_2:=(i_3-i_1)+\dots+(i_{2\lfloor \frac{t-1}{2}\rfloor+1}-i_{2\lfloor \frac{t-1}{2}\rfloor-1})+(i_4-i_2)+\dots+(i_{2\lfloor \frac{t}{2}\rfloor}-i_{2\lfloor \frac{t}{2}\rfloor-2})\leq 2M-4.$$
		We have that
		$$ \frac{(2t-3)(2t-2)}{2}=1+2+\cdots +(2t-3)\leq M_1+M_2\leq 3M-5,$$
		it follows that
		$$ t\leq \sqrt{\frac{3M}{2}-\frac{39}{16}}+\frac{5}{4}.$$$\hfill\qed$ 
		
	\end{proof}
Indeed, we do not know a sharp version for this result. We checked a possible bound using  a computer and we conjecture  that there exists a constant $k>0$ such that 
	$$|\{1\leq i\leq p-2: \gamma^{i+1}= ic+d \}|<k\cdot \log(p).$$
	
	\begin{proposition}\label{item33}
		For $p$ an odd prime and $n>1$ an integer, let $\Fq$ be a finite field with $q=p^n$ elements and let $\gamma\in\Fq\backslash\{1\}$. Then
		$$|\{1\leq i\leq q-2: \gamma^{i+1}= i(1-\gamma)+1 \}|\leq\, \frac{q}{p}+\sqrt{\frac{3p}{2}-\frac{39}{16}}+\frac{1}{4}.$$
	\end{proposition}
	
	\begin{proof}
		Assume that $\gamma\in\Fq\backslash\Fp$ and let $l=ord_{\Fq}(\gamma)$. In order to prove our result, we consider two cases: $l>p$ and $l<p$.
		
		We assume that $l>p$ and note that $\gamma^{i+1}$ has period $l$ and $i(1-\gamma)+1$ has period $p$. Let $f(i):=\gamma^{i+1}-i(1-\gamma)-1$. We recall that $\gcd(l,p)=1$. By Corollary \ref{item31}, the number of roots of $f(i)$ in $\big[1, lp\big\lfloor\frac{q-2}{lp}\big\rfloor\big]$ is at most $p\big\lfloor\frac{q-2}{lp}\big\rfloor$. In the interval $[lp\big\lfloor\frac{q-2}{lp}\big\rfloor+1, q-2\big]$, the number of roots of $f(i)$ is at most $p$, since $q-2-lp\big\lfloor\frac{q-2}{lp}\big\rfloor<lp$. We split the problem into the following subcases:
		\begin{itemize}
			\item Assume that $n=2$. Since $l>p$, 
			$$\left\lfloor\frac{q-2}{lp}\right\rfloor+p=\left\lfloor\frac{p^2-2}{lp}\right\rfloor+p=p=\frac{q}{p}.\\[.5em]$$
			\item Suppose that $n=3$. Since $(p+1)\nmid (p^3-1)$, then $l$ is at least $p+2$. Thus
			$$p\left\lfloor\frac{q-2}{lp}\right\rfloor+p=p\left\lfloor\frac{p^3-2}{lp}\right\rfloor+p<p\frac{p^3-2}{lp}+p\leq \frac{p^3-2}{p+2}+p\leq p^2 = \frac{p^3}{p}=\frac{q}{p}.\\[.5em]$$
			\item Assume that $n\geq4$. Since $l\geq p+1$, we have that
			$$p\left\lfloor\frac{q-2}{lp}\right\rfloor+p=p\left\lfloor\frac{p^n-2}{lp}\right\rfloor+p<p\frac{p^n-2}{lp}+p\leq \frac{p^n-2}{p+1}+p\leq p^{n-1} = \frac{p^n}{p}=\frac{q}{p}.\\[.5em]$$
		\end{itemize}
		
		Therefore our result is proved for $l>p$. Now, we assume $l<p$ and observe that $\gamma^{i+1}= i(1-\gamma)+1$ is the same as $\gamma^i+\ldots+\gamma+1=-i$. We define $f(i)=\sum_{j=0}^{i}\gamma^j$ and $g(i)=-i$. By Lemma \ref{item69} we have that
		$$\begin{aligned}
		|\{i\in[1, lp]: f(i)=g(i)\}| &= |\{f(i):0\leq i \leq l \}\cap \{g(i):0\leq i \leq p \}|\\
		&= |\{f(i):0\leq i \leq l \}\cap \{-i:0\leq i \leq p \}|\\
		&=|\{f(i):0\leq i \leq l \}\cap \Fp|.
		\end{aligned}$$
		
		If   $k$ is a value  such that $0\leq k \leq l-1$ and  $f(k),f(k+1)\in \Fp$, then
		\begin{equation}\label{equa53}
		f(k)=\gamma^k+\ldots+\gamma+1=c_k\in\Fp;
		\end{equation}
		\begin{equation}\label{equa54}
		f(k+1)=\gamma^{k+1}+\ldots+\gamma+1=c_{k+1}\in\Fp.
		\end{equation}
		
		Since $c_k\neq0$ and $k<l=ord_{\Fq}(\gamma)$, the Equations \eqref{equa53} and \eqref{equa54} imply that
		$$\gamma=\frac{c_{k+1}-1}{c_k}\in\Fp,$$
		which is a contradiction. Therefore, if $f(k)\in
\Fp$, then $f(k+1)\notin
\Fp$. Then we have an upper bound for the number of elements in $\{f(i):0\leq i \leq l \}\cap \Fp$ given by
		$$|\{i\in[1, lp]: f(i)=g(i)\}|=|\{f(i):0\leq i \leq l \}\cap \Fp|\leq\bigg\lfloor \frac{l}{2} \bigg\rfloor.$$
		
		Therefore the number of roots of $f(i)-g(i)$ in $\big[1, lp\big\lfloor\frac{q-2}{lp}\big\rfloor\big]$ is at most $\big\lfloor\tfrac{l}{2}\big\rfloor\big\lfloor\frac{q-2}{lp}\big\rfloor$ and the number of roots in $[lp\big\lfloor\frac{q-2}{lp}\big\rfloor+1, q-2\big]$ is at most $\big\lfloor\tfrac{l}{2}\big\rfloor$ since $q-2-lp\big\lfloor\frac{q-2}{lp}\big\rfloor<lp$. Then
		$$|\{1\leq i\leq q-2: \gamma^{i+1}= i(1-\gamma)+1 \}|\leq\bigg\lfloor\frac{l}{2}\bigg\rfloor\bigg\lfloor\frac{q-2}{lp}\bigg\rfloor+\bigg\lfloor\frac{l}{2}\bigg\rfloor<\frac{q}{2p}+\frac{p}{2}\leq\,\frac{q}{p},$$
	and the result is proved in the case where $\gamma\in\Fq\backslash\Fp$.
	
	Assume that $\gamma\in\Fp$. We have that $l:=ord_{\Fq}(\gamma)$ divides $p-1$ and then, by Corollary \ref{item31}, the number of elements $i\in[p-1,\,q-2]$ for which $\gamma^{i+1}= i(1-\gamma)+1$ is at most $\tfrac{q}{p}-1$. It follows from Lemma \ref{item32} that the number of solutions $i\in[1,p-2]$ of the equation $\gamma^{i+1}= i(1-\gamma)+1$ is bounded by
		$$\sqrt{\frac{3p}{2}-\frac{39}{16}}+\frac{5}{4},$$
	and then our result follows. $\hfill\qed$ 
	\end{proof}

	\section{The Main Results}
	
For polynomials with Carlitz rank $1$, the weight of $f$ is well determined as it is shown in the following proposition.

\begin{proposition}\label{item24}
	Let $\Fq$ be a finite field with odd characteristic $p$ and let $f$ be a permutation polynomial of $\Fq$ with $Crk(f)=1$. Then 
	$$ \omega(f)\in\{1,2,q-\tfrac{q}{p},q-\tfrac{q}{p}-1\}.$$
\end{proposition} 

\begin{proof}
	Since $Crk(f)=1$, there exist $a_0\in\Fq^*$ and $a_1,a_2\in\Fq$ such that
	$$f(x)=(a_0x+a_1)^{q-2}+a_2.$$
	From the Binomial Theorem we have that
	\begin{equation}\label{item94}
	(a_0x+a_1)^{q-2}+a_2=a_2+\sum_{i=0}^{q-2}\binom{q-2}{i}a_0^i a_1^{q-2-i}x^i.
	\end{equation}
	If $a_1=0$, it follows that $\omega(f)=1$ if $a_2=0$ and $\omega(f)=2$ if $a_2\neq0$. If $a_1\neq0$, then from the Equation~\eqref{item94} it follows that
	$$(a_0x+a_1)^{q-2}+a_2=a_2+a_1^{q-2}\sum_{i=0}^{q-2}\binom{q-2}{i}\left(a_0^{-1} a_1\right)^{-i}x^i.$$
	and therefore
	$$\omega(f)=\begin{cases}
	q-s-1,&a_2\neq -a_1^{q-2};\\
	q-s-2,&a_2=-a_1^{q-2},\\
	\end{cases}$$
	where $s$ is the number of $i$ with $1\leq i\leq q-2$ such that
	$$\binom{q-2}{i}\equiv 0\pmod{p}.$$
	From Lucas' congruence it follows that
	$$\binom{q-2}{i}\equiv (i+1)(-1)^i\pmod{p}.$$
	Therefore $s=\tfrac{q}{p}-1$ and then our result follows.$\hfill\qed$ 
\end{proof}

\begin{remark}
	The cases where $Crk(f)=1$ and $\omega(f)\in\{1,2\}$ are not considered in Theorem \ref{item90}. Indeed, for the remaining cases, Proposition \ref{item24} asserts that $\omega\ge q-\tfrac{q}{p}-1$ if $Crk(f)=1$, which is better than the bound $q/2-2$ provided by Theorem \ref{item90}. 
\end{remark}
	
In the case where $f$ is a polynomial with Carlitz rank is $2$, Theorem \ref{item90} implies that
	$$\omega(f)>\frac{q}{3}-2.$$
	
	Our main result improves this lower bound in the case where the Carlitz rank of $f$ is $2$. 
	
	\begin{theorem}\label{item34} Let $\Fq$ be a finite field with odd characteristic $p$ and let $f$ be a permutation polynomial of $\Fq$ with $Crk(f)=2$. Then
		$$ \omega(f)\geq q-\frac{q}{p}-\sqrt{\frac{3p}{2}-\frac{39}{16}}+\frac{1}{4}.$$
		
	\end{theorem}
	
	\begin{proof}
		By Lemma \ref{item30} and Remarks \ref{item40} and \ref{item39}, we can assume that $f$ is given by
		\begin{equation}\label{equa60}
		f(x)=a_2^{-1}\sum\limits_{i=1}^{q-2}{x^i\big[(a_1-ia_2^{-1})(a_1+a_2^{-1})^{q-2-i} - a_1^{q-1-i}\big]},
		\end{equation}
		where $a_1\in\Fq$ and $a_2\in\Fq^*$. We split the proof into the following cases:
		
		\begin{enumerate}[(a)]
			\item If $a_1=0$, from Equation \ref{equa60} it follows that
			$$ f(x)= a_2^{-1}\sum\limits_{i=1}^{q-2}{x^i\big[-i(a_2^{-1})^{q-1-i}\big]}=-a_2^{-1}\sum\limits_{i=1}^{q-2}{i x^i a_2^i}.$$
			We observe that $i\,a_2^{i}=0$ if and only if $i\equiv 0\pmod{p}$. Therefore
			$$ \omega(f)= q-2-\bigg(\frac{q}{p}-1\bigg) = q-\frac{q}{p}-1.$$
			\item Assume that $a_1+a_2^{-1}=0$ and $a_1\neq 0$. In this case $f$ can be rewritten as
			$$ f(x)= -a_2^{-1}\sum\limits_{i=1}^{q-2}{x^i\, a_1^{q-1-i}}.$$
			and therefore $\omega(f) = q-2$.
			\item Suppose that $a_1\neq 0$ and $a_1+a_2^{-1}\neq 0$. In this case,
			$$\begin{aligned} f(x)&=a_2^{-1}\sum\limits_{i=1}^{q-2}{x^i\big[(a_1-ia_2^{-1})(a_1+a_2^{-1})^{q-2-i} - a_1^{q-1-i}\big]}\\[.5em]
			&=a_2^{-1}\sum\limits_{i=1}^{q-2}{x^i\big[(a_1-ia_2^{-1})(a_1+a_2^{-1})^{-(i+1)} - a_1^{-i}\big]}\\[.5em]
			&=a_2^{-1}\sum\limits_{i=1}^{q-2}{x^ia_1(a_1+a_2^{-1})^{-(i+1)}\left[1-i\left(\dfrac{a_1+a_2^{-1}}{a_1}-1\right) - \left(\dfrac{a_1+a_2^{-1}}{a_1}\right)^{i+1}\right]}.\\[.5em]
			\end{aligned}$$
			
			Let $\gamma := \frac{a_1+a_2^{-1}}{a_1}$. By hypothesis, we have that $\gamma\not\in\{0,1\}$ and then we only need to compute the number of solutions of the equation
			$$\gamma^{i+1}= i(1-\gamma)+1\text{ with }i\in[1,q-2].$$
			Therefore our result follows from Preposition \ref{item33}.$\hfill\qed$
		\end{enumerate}
		
	\end{proof}
	
\begin{corollary}\label{item89}  
	Let $p$ be an odd prime and let $\Fq$ be a finite field with characteristic $p$. Set 
	\begin{equation}\label{eq05}
	\nu_p=\max\limits_{\gamma\in\Fp\backslash\{1\}}|\{1\leq i\leq p-2: \gamma^{i+1}= i(1-\gamma)+1 \}|.
	\end{equation}
	 If $f$ is a permutation polynomial of $\Fq$ with $Crk(f)=2$, then
	$$ \omega(f)\geq q-\frac qp-1-\nu_p.$$	
	Furthermore, $0\leq \nu_p\leq \sqrt{3p/2-39/16}+5/4$.
\end{corollary}

\begin{remark}
	It follows from the proof of Theorem~\ref{item34} that the inequality in Corollary~\ref{item89} is sharp, i.e. for all $n\in\Z_{+}^*$ there exists a permutation polynomial $f(x)\in\Fpn[x]$ with $Crk(f)=2$ and $\omega(f)= p^n-p^{n-1}-1-\nu_p$.
\end{remark}

	From here, an open question is compute the exact value of $\nu_p$. In fact, this kind of question is interesting from the cryptography point of view and it have been studied in \cite{coppersmith2000polynomial,winterhof2002polynomial}.
	
	\begin{example}
		It is easy to verify that $\nu_{11}=3$, where the maximum value in \eqref{eq05} is reached by $\gamma=7$.  For each positive integer $n$, we set
		$$f_n(x)=\sum\limits_{i=1}^{11^n-2}{\big[4^{i+1}(2-i) - 6^i\big]x^i}.$$
We note that $f_n$ is a permutation polynomial with Carlitz rank $2$ in $\F_{11^n}$ since  $f_n$ has been chosen using Equation \eqref{equa60} and the fact that $\gamma=\frac{a_1+a_2^{-1}}{a_1}$. The polynomial $f_n$ can also be seen as
		$$f_n(x)\equiv((2-x)^{11^n-2}+1)^{11^n-2}-8\pmod{x^{11^n}-x}.$$
		By the proof of Theorem \ref{item34} we know that $\omega(f_n)=11^n-11^{n-1}-4$. In addition, any permutation polynomial $g(x)$ with Carlitz rank $2$ over $\F_{11^n}$ satisfies $\omega(g)\geq 11^n-11^{n-1}-4.$
	\end{example}

\section{Acknowledgments}

We are very grateful to the anonymous  referees for careful reading of the paper and valuable
suggestions and comments. This study was financed in part by the Coordenação de Aperfeiçoamento de Pessoal de Nível Superior - Brasil (CAPES) - Finance Code 001. 
	
\bibliographystyle{siam}
\bibliography{biblio}
	
\end{document}